\documentclass{amsart}
\usepackage{amsmath}               
\usepackage{amsfonts}  
 
\usepackage{amsthm}  
\usepackage{amssymb}               
\usepackage[margin=1.5in]{geometry} 
\usepackage{cancel}
\usepackage[shortlabels]{enumitem} 
\usepackage[dvipsnames]{xcolor}
\usepackage{mathrsfs}
\usepackage{ragged2e}
\usepackage{graphicx}
\usepackage{tikz-cd}
\usepackage{float}
\usepackage{stmaryrd}

\usepackage[backref]{hyperref}
\hypersetup{
  colorlinks   = true,          
  urlcolor     = Green,          
  linkcolor    = blue,          
  citecolor   = BurntOrange          
}

\usepackage{url}

\makeatletter
\def\blfootnote{\xdef\@thefnmark{}\@footnotetext}
\makeatother

\DeclareMathOperator{\Vor}{Vor}

\DeclareMathOperator{\Sym}{Sym}
\DeclareMathOperator{\Pic}{Pic}
\DeclareMathOperator{\Div}{Div}

\DeclareMathOperator{\divi}{div}

\DeclareMathOperator{\Spec}{Spec}
\DeclareMathOperator{\Spf}{Spf}

\DeclareMathOperator{\chara}{char}
\DeclareMathOperator{\ord}{ord}

\newtheorem{theorem}{Theorem}
\newtheorem{maintheorem}{Theorem}

\newtheorem{lemma}[theorem]{Lemma}
\newtheorem{proposition}[theorem]{Proposition}

\theoremstyle{definition} 
\newtheorem{definition}[theorem]{Definition}
\newtheorem{example}[theorem]{Example}

\theoremstyle{remark} 
\newtheorem{remark}[theorem]{Remark}

\numberwithin{theorem}{section}

\newcommand{\bij}{\stackrel{\sim}{\longrightarrow}}

\title{A Lefschetz Hyperplane Theorem for non-Archimedean Jacobians}
\author{Tif Shen}
\date{\today}

\address{Department of Mathematics, Yale University}
\email{jif.shen@gmail.com}

\begin{document} 
\maketitle \vspace{-25pt} 
\begin{abstract} We establish a Lefschetz hyperplane theorem for the Berkovich analytifications of Jacobians of curves over an algebraically closed non-Archimedean field. Let $J$ be the Jacobian of a curve $X$, and let $W_d \subset J$ be the locus of effective divisor classes of degree $d$. We show that the pair $(J^{an},W_d^{an})$ is $d$-connected, and thus in particular the inclusion of the analytification of the theta divisor $\Theta^{an}$ into $J^{an}$ satisfies a Lefschetz hyperplane theorem for $\mathbb{Z}$-cohomology groups and homotopy groups. A key ingredient in our proof is a generalization, over arbitrary characteristics and allowing arbitrary singularities on the base, of a result of Brown and Foster for the homotopy type of analytic projective bundles. 
\end{abstract}

\blfootnote{This research was partially supported by NSF grant CAREER DMS-1149054 (PI: Sam Payne).}

%

\section{Introduction} 
The main objective of this paper is to establish, for Berkovich analytifications of Jacobians of curves, a Lefschetz hyperplane theorem for $\mathbb{Z}$-cohomology and homotopy groups. 
Let $K$ be an algebraically closed field complete with respect to a non-trivial non-Archimedean norm $|\cdot|_K$. Let $X$ be a smooth projective curve over $K$ of genus $g$, and let $J$ be its Jacobian. 
Fix a basepoint in $X$, the Abel-Jacobi map realizes the locus $W_d \subset \Pic^d(X)$ of effective divisor classes of degree $d$ as a subset of $J$, which gives us an inclusion $W_d^{an} \subset J^{an}$.

\begin{maintheorem}[Lefschetz for analytic Jacobians]\label{nonarch wd} For $1 \leq d\leq g-1$, the pair $(J^{an},W_d^{an})$ is $d$-connected, i.e.
	\[\pi_i(J^{an},W_d^{an})=0\mathrm{\ for\ }i \leq d.\]
	
\end{maintheorem}

\noindent{}In particular, the inclusion $W_d^{an} \hookrightarrow J^{an}$ induces isomorphisms between $\mathbb{Z}$-cohomology groups of dimension $<d$, and an injection in dimension $d$.

Let $D$ denote an ample divisor on a smooth projective variety $Y$ of dimension $n$ over $K$. As a consequence of the $\ell$-adic Lefschetz hyperplane theorem \cite[\S 4.1.6]{Del80} and of Berkovich's weight 0 comparison theorem \cite[Theorem 1.1]{Ber00}, the inclusion $D^{an} \hookrightarrow Y^{an}$ induces isomorphisms between $\mathbb{Q}$-cohomology groups of dimensions $< n-1$, and an injection in dimension $n-1$. However, Payne noticed that the previous statement does not hold in general if we replace $\mathbb{Q}$ with $\mathbb{Z}$, or if we replace cohomology groups with homotopy groups \cite[Example 16]{Pay15}. Let $\Theta$ denote the theta divisor of $J$. In the same paper, Payne suggested that the pair $(J^{an},\Theta^{an})$ may satisfy a Lefschetz hyperplane theorem for $\mathbb{Z}$-cohomology and homotopy groups \cite[Example 15]{Pay15}. Since $W_{g-1}$ is a translate of $\Theta$, as a special case of Theorem \ref{nonarch wd}, we show that that a Lefschetz theorem for $\mathbb{Z}$-cohomology and homotopy groups indeed holds for the pair $(J^{an},\Theta^{an})$. 

In general, the Poincar\'e formula gives us the following equality of fundamental classes
\[[W_d] = \frac{1}{(g-d)!}\bigcap^{g-d}[\Theta].\] 
As subvarieties of the Jacobian, $W_d$ is contained in $W_{d+1}$. It thus follows from the Poincar\'e formula that $W_d$ is ample in $W_{d+1}$: the restriction of $\Theta$ is ample and positive multiples of an ample divisor are ample. By \cite[Corollary IV.4.5]{ACGH85}, the singular locus of $W_{d+1}$ is equal to the locus $W^1_{d+1}$ of divisor classes degree degree $d+1$ and of rank at least $1$, and the latter is contained in $W_d$ (see \cite[p. 250]{GH11}) . Over $\mathbb{C}$, the classical Lefschetz hyperplane theorem \cite[Theorem 7.4]{Mil63} implies that the pair $(W_{d+1},W_d)$ is $d$-connected. Theorem \ref{nonarch wd} can therefore also be viewed as a non-Archimedean analog of the Lefschetz hyperplane theorem for the pair $(W^{an}_{d+1},W^{an}_d)$.

\subsection{Tropicalization} The main technical step in the proof Theorem \ref{nonarch wd} consists in showing that the natural map from $W^{an}_d$ to its tropicalization is a homotopy equivalence. Recall the analytification $X^{an}$ has a skeleton $\Gamma$, which is a metric graph. We know from \cite{Ber90,BR13} that the Jacobian $J(\Gamma)$ of $\Gamma$ is the skeleton of $J^{an}$. In particular, there is a canonical retraction map from $J^{an}$ onto $J(\Gamma)$. Let $W_d(\Gamma)$ denote the image in $J(\Gamma)$ of $W^{an}_d$.

\begin{maintheorem}\label{tropicalize wd} The map $W^{an}_d \rightarrow W_d(\Gamma)$ is a homotopy equivalence for all $d$.
\end{maintheorem}

\noindent{}It suffices then to show that the pair $(J(\Gamma),W_d(\Gamma))$ is $d$-connected. Since the retraction of $J^{an}$ onto $J(\Gamma)$ is compatible with the Abel-Jacobi maps of $X^{an}$ and $\Gamma$ \cite[Proposition 6.1]{BR13}, $W_d(\Gamma)$ is identified with the locus of effective divisor classes of degree $d$ on $\Gamma$. In particular, there is a natural surjection $\Sym^d(\Gamma) \rightarrow W_d(\Gamma)$, which we show to be a homotopy equivalence. The desired statement for the pair $(J(\Gamma),W_d(\Gamma))$ then follows by comparing the homotopy groups of $\Sym^d(\Gamma)$ with the homotopy groups of $J(\Gamma)$ (see \S \ref{sec trop sym}). 


\subsection{Morphisms with contractible fibers} Let $\Sym^{d,an}(X)$ denote the analytification of the $d$-th symmetric product $\Sym^d(X)$ of $X$. A crucial part of our proof of Theorem \ref{tropicalize wd} consists in showing the natural map $\Sym^{d,an}(X) \rightarrow W_d^{an}(X)$ is also a homotopy equivalence. Since the fiber over a divisor class $[D] \in W_d(X)$ is the projective space $|D|$, we obtain this homotopy equivalence as a special case of the following theorem. Let $K$ be a (not necessarily algebraically closed) field complete with respect to a non-trivial non-Archimedean norm.

\begin{maintheorem}\label{projective fibers} Let $f: X\rightarrow Y$ be a surjective morphism of projective $K$-varieties. Suppose that there is a finite stratification $Y = \coprod_i Y_i$ such that $f:X\times_Y Y_i \rightarrow Y_i$ is a projective bundle of rank $r_i$ over $Y_i$. Then, there is a finite extension $K\subset L$ such that $f^{an}_L:(X_L)^{an} \rightarrow (Y_L)^{an}$ is a homotopy equivalence.
	
Moreover, if we suppose the field $K$ has a countable dense subset, then we can take $L = K$.\end{maintheorem}

\noindent{}Brown and Foster have shown that over $K=\mathbb{C}(\!(t)\!)$, if $f:X \rightarrow Y$ is a projective bundle with $Y$ smooth, then $f^{an}: X^{an} \rightarrow Y^{an}$ is a homotopy equivalence \cite[Corollary 1.1.3]{BF14}. Their argument follows the minimal model approach developed in \cite{dFKX12,MN15,NX13}. The assumptions of equi-characteristic $0$ and of $Y$ being smooth (or mildly singular) are essential for this method.

We follow a different approach, which works over $K$ of arbitrary characteristic and allows arbitrarily bad singularities for $Y$. Recall from \cite{HLP14} that over $K$ with a countable dense subset, the analytification $X^{an}$ of a quasi-projective $K$-scheme of dimension $d$ embeds into $\mathbb{R}^{2d+1}$; in particular $X^{an}$ is metrizable and has a countable dense subset. Therefore, we can apply the Vietoris-Begle-Smale mapping theorem \cite[Main Theorem]{Sma57} to show that over such $K$, any proper surjection $f^{an}:X^{an} \rightarrow Y^{an}$ is a homotopy equivalence if the fibers of $f^{an}$ are contractible. We conclude the proof of Theorem \ref{projective fibers} with a spreading out argument (see \S \ref{secanal}).



\begin{remark} One could ask if the conclusion of Theorem \ref{projective fibers} holds without having to pass to a finite extension $L$ of $K$. We are not aware of examples of morphisms $f$ satisfying the hypothesis of the theorem for which $f^{an}$ fails to be a homotopy equivalence. The finite extension is used only to apply \cite[Theorem 14.2.3]{HL16}.
\end{remark}

\noindent{}This paper is structured as follows: In \S\ref{sectiontwo}-\ref{sec trop sym} we review the tropical Abel-Jacobi theory of metric graphs and establish a tropical Lefschetz hyperplane theorem for Jacobians of metric graphs. In \S \ref{secanal} we survey the construction of Berkovich analytifications and prove Theorem \ref{projective fibers}. In \S \ref{sec skel}-\ref{sec sym} we show that for analytifications of curves, the symmetric product of the skeleton is the skeleton of the symmetric product. Finally, in \S \ref{sectionabeljac}-\ref{finalsec} we establish Theorem \ref{nonarch wd} and Theorem \ref{tropicalize wd} by combining all of the previous results.

\section{Metric Graphs and their Jacobians}\label{sectiontwo}
A \textbf{metric graph} $\Gamma$ is the geometric realization of a graph $G=(V,E)$ equipped with an edge-length function $\ell:E\rightarrow \mathbb{R}_{>0}$. Each edge $e$ is identified with a line segment in $\Gamma$ of length $\ell(e)$. Recall from \cite{BF06,MZ08} that a harmonic $1$-form on $\Gamma$ is given by assigning a real-valued slope to each edge in $\Gamma$ such that the sum of the incoming slopes is zero at every vertex. Let  $\Omega(\Gamma)$ denote the space of harmonic $1$-forms on $\Gamma$, and let  $\Omega^*(\Gamma)$ be its dual. The \textbf{Jacobian} of $\Gamma$ is defined as the quotient $$J(\Gamma) := \Omega^*(\Gamma)/H_1(\Gamma,\mathbb{Z})$$
by realizing $H_1(\Gamma,\mathbb{Z})$ as a lattice in $\Omega^*(\Gamma)$ via integration over $1$-cycles. We refer the reader to \cite{BF06,MZ08} for the details of the above constructions. There is a canonical identification between $\Omega^*(\Gamma)$ and $H_1(\Gamma,\mathbb{R})$ (see \cite[Lemma 2.1]{BF06}). In particular, the Jacobian can be equivalently defined as the torus $$J(\Gamma):= H_1(\Gamma,\mathbb{R})/H_1(\Gamma,\mathbb{Z}).$$

Fix a full-rank lattice $\Lambda$ in a real vector space $V$. Given a positive definite quadratic form $Q$ on $V$, the Voronoi polytope $\Vor(Q)$ associated to $Q$ is the set of points \[\Vor(Q):=\{x\in V:
Q(x) \leq Q(x-\lambda)\ \forall\lambda\in
\Lambda\}.\] 
Set $\Lambda = H_1(\Gamma,\mathbb{Z})$, and consider the positive definite quadratic form $Q(\Gamma)$ on $H_1(\Gamma,\mathbb{R})$ given by
$$Q(\Gamma)\bigg(\sum_{e\in E}a_e e\bigg):= \sum_{e\in E}a_e^2 \ell(e)$$
where $\sum_{e\in E}a_e e$ is a 1-chain. The \textbf{theta divisor} $\Theta(\Gamma)$ of $J(\Gamma)$ is the image in $J(\Gamma)$ of the codimension 1 skeleton of the associated Voronoi polytope $\Vor(Q(\Gamma))$.

\subsection{Divisors of a metric graph} We now review the theory of divisors on a metric graph. For further details and references, see \cite{BF11,MZ08}. A \textbf{divisor} is a finitely supported element $D= \sum_{x\in \Gamma}D_x  x$ of the free abelian group on $\Gamma$. Let $f: \Gamma \rightarrow \mathbb{R}$ be a piecewise linear function with integral slopes, let $\ord_x(f)$ denote the sum of outgoing slopes of $f$ at $x$, and let $\divi(f)$ be the divisor defined by the sum
\[\divi(f):=\sum_{x\in \Gamma}\ord_x(f)x.\]
Two divisors $D$ and $D'$ are said to be \textbf{equivalent} if $D-D' =\divi(f)$ for some $f$. 

Mikhalkin and Zharkov constructed, for each choice of basepoint $p\in \Gamma$, a
\textbf{tropical Abel-Jacobi map} $\alpha_p: \Gamma \rightarrow J(\Gamma)$. The map $\alpha_p$ is defined by sending a point $q\in \Gamma$ to the integral $\int^q_p \in \Omega^*(\Gamma)$. Let $\deg(D):= \sum_{x\in \Gamma} D_x$ be the degree of a divisor $D$; let $\Div^d(\Gamma)$ be the set of divisors of degree $d$. Extending $\alpha_p$ linearly gives maps $\alpha_{p,d}:\Div^d(\Gamma)\rightarrow J(\Gamma)$. Note that the tropical Abel-Jacobi map takes equivalent divisors to the same point in $J(\Gamma)$. Let $\Pic^d(\Gamma)$ be the equivalent classes of divisors of degree $d$. Then in particular, we get an induced map $\alpha_{p,d}: \Pic^d(\Gamma) \rightarrow J(\Gamma)$.

\begin{theorem}[{\cite[Theorem 6.2]{MZ08}}]The map $\alpha_{p,d}: \Pic^d(\Gamma) \rightarrow J(\Gamma)$ is a bijection, and when $d=0$ it does not depend on the choice of basepoint $p$.
\end{theorem}

\subsection{Effective divisors} A divisor is \textbf{effective} if all of its coefficients are non-negative; a divisor class is effective if it contains an effective representative. Let $W_d(\Gamma) \subset J(\Gamma)$ denote the image under $\alpha_{p,d}$ of the locus of effective divisor classes in $\Pic^d(\Gamma)$. For $d' \leq d$, one can easily see that $W_{d'}(\Gamma) \subset W_d(\Gamma)$. 

For $1\leq d$, let $\Theta_d(\Gamma)$ denote the image in $J(\Gamma)$ on the $d$-skeleton of $\Vor(Q(\Gamma))$. Let $b$ denote the \textbf{genus} of $\Gamma$, i.e. its first Betti number. By \cite[Corollary 8.6]{MZ08}, $W_{b-1}(\Gamma)$ is a translate of $\Theta_{b-1}(\Gamma)=\Theta(\Gamma)$.

\begin{example}\label{wd cube} Let $\Gamma$ be a bouquet of $n$ circles of arbitrary lengths. Let $e_1,...,e_n$ denote the cycles in $H_1(\Gamma,\mathbb{Z})$ defined by the edges of $\Gamma$. Then $\Vor(Q(\Gamma))$, up to translation by $\kappa = \frac{1}{2}\sum_{1\leq i \leq n} e_i$, is the cube with vertex set $\{ \sum_{i\in I}e_i\} $ as $I$ ranges over the subsets of $\{1,...,n\}$. From this, one can easily show that $W_1(\Gamma) = \Theta_1(\Gamma) + \kappa$. Since $W_d(\Gamma) = W_1(\Gamma)+\cdots+W_1(\Gamma)$, it follows that $W_d(\Gamma) = \Theta_d(\Gamma) + \kappa$ for all $d$, and therefore $J(\Gamma)$ can be obtained from $W_d(\Gamma)$ by attaching cells of dimensions $>d$. \end{example}

\noindent{}In general, $W_d(\Gamma)$ is not necessarily a translate of $\Theta_d(\Gamma)$. In fact, they are not necessarily homotopic to each other, as the next example shows.
\begin{example}\label{wddifficult} Consider the metric graph $\Gamma$ on the left on Figure ~\ref{skelgapexample}. The front and the back of its Voronoi polytope $\Vor(Q(\Gamma))$ are displayed on the center and right respectively; on the images, each distinct vertex of $\Vor(Q(\Gamma))$ is labelled with a unique letter. The quotient map $\Vor(Q(\Gamma)) \rightarrow J(\Gamma)$ identifies the opposite faces of $\Vor(Q(\Gamma))$, and sends the $d$-skeleton of $\Vor(Q(\Gamma))$ onto $\Theta_d(\Gamma)$. In the pictures below, we colored the faces that get identified with the same color. From this, one can deduce that $\Theta_1(\Gamma)$ is a graph with 4 vertices and 9 edges, and therefore has genus 6. However, $W_1(\Gamma)$ is homotopic to $\Gamma$, which has genus 3, and therefore cannot be identified with $\Theta_1(\Gamma)$.

\begin{figure}[h]
			\includegraphics[scale=.4]{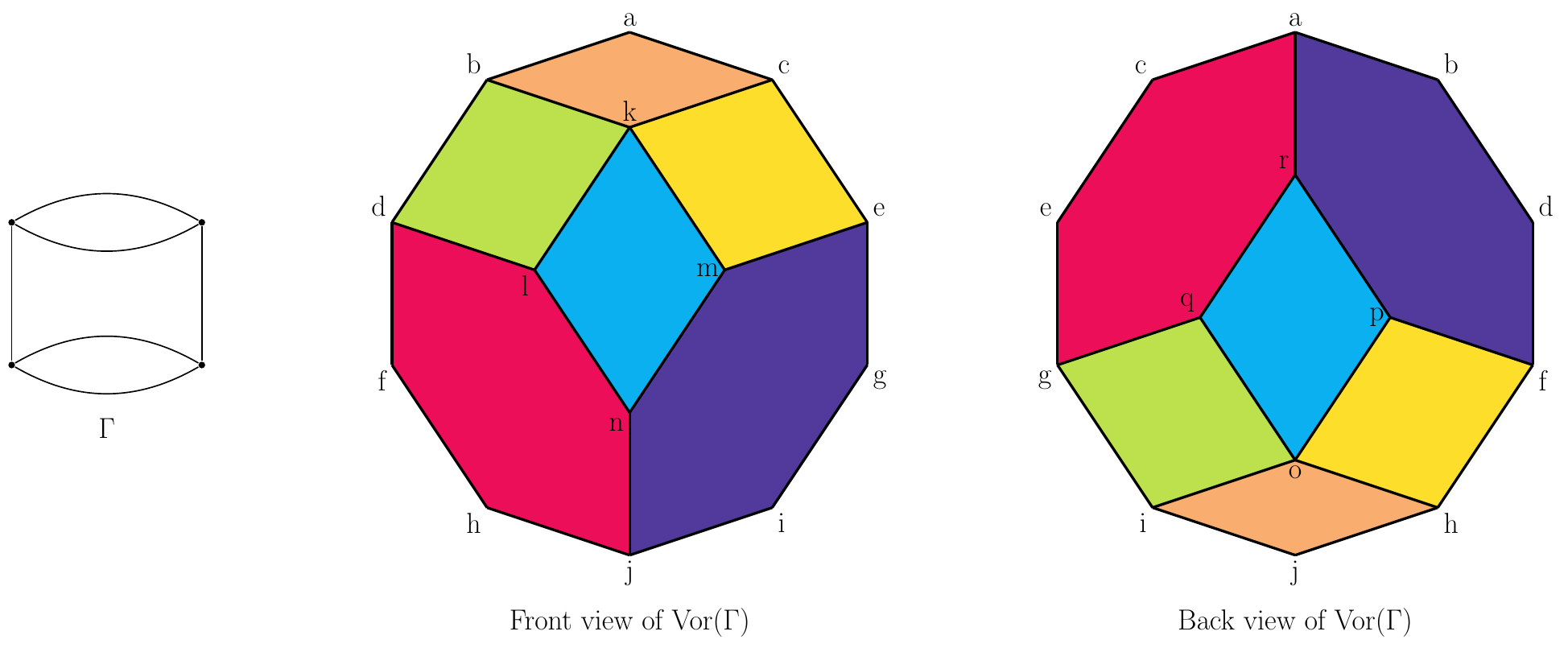}
			\caption{}
			\label{skelgapexample}
\end{figure}

\end{example}


\section{Symmetric Products and Lefschetz for $W_d(\Gamma)$} \label{sec trop sym}

Given a space $A$ and a subspace $B\subset A$, recall that the pair $(A,B)$ is \textbf{$n$-connected} if $\pi_i(A,B)=0$ for $i \leq n$. By construction, $J(\Gamma)$ can be obtained from $\Theta_d(\Gamma)$ by attaching cells of dimensions $>d$. In particular, the pair $(J(\Gamma),\Theta_d(\Gamma))$ is $d$-connected. In this section, we show that the same is true for the pair $(J(\Gamma),W_d(\Gamma))$, which gives us a tropical analog of the Lefchetz hyperplane theorem for $W_d(\Gamma)$.
\begin{theorem}[Lefschetz for $W_d(\Gamma)$]{\label{tropical wd}} For $d \geq 1$, the pair $(J(\Gamma),W_d(\Gamma))$ is $d$-connected.
\end{theorem}
\noindent{}As noted in the introduction, the key technical step of the proof of Theorem \ref{tropical wd} consists in showing that the natural map $\Sym^d(\Gamma) \rightarrow W_d(\Gamma)$ is a homotopy equivalence. First, let us recall some basic facts about symmetric products of topological spaces.



\subsection{Symmetric products} Given a topological space $\Delta$, the $d$-th symmetric group $S_d$ acts on the product $\Delta^d$ by permuting the factors. The \textbf{$d$-th symmetric product} $\Sym^d(\Delta)$ of $\Delta$ is the quotient space $\Delta^d/S_d$. The points of $\Sym^d(\Delta)$ can be written as formal unordered sums $p_1+\cdots+p_d$ of points $p_1,...,p_d \in \Delta$. As noted in \cite[\S 4K]{Hat02}, taking symmetric products preserves homotopy equivalence in following sense. A continuous map $f:\Delta \rightarrow \Delta'$ induces a continuous map $\Sym^d(f):\Sym^d(\Delta) \rightarrow \Sym^d(\Delta')$ defined by $p_1+\cdots+p_d \mapsto f(p_1)+\cdots+f(p_d)$. If $f$ is a homotopy equivalence, then so is $\Sym^d(f)$.

\begin{example}\label{exsim} By the fundamental theorem of symmetric functions, there is a canonical homeomorphism $\Sym^d(\mathbb{C}) \bij \mathbb{C}^d$, which sends $\Sym^d(\mathbb{C}^*) \subset \Sym^d(\mathbb{C})$ onto the subspace $\mathbb{C}^*\times \mathbb{C}^{d-1}$ (see \cite[Example 4K.4]{Hat02}). Since $S^1$ is homotopic to $\mathbb{C^*}$, we have that $\Sym^d(S^1)$ is homotopic to $\mathbb{C}^*\times \mathbb{C}^{d-1}$, and the latter is homotopic to $S^1$.
\end{example}

\noindent{}Set $\Delta = \Gamma$. Since $\Gamma$ is homotopic to a wedge sum $\bigvee^{b} S^1$, we have a homotopy equivalence from $\Sym^d(\Gamma)$ to $\Sym^d(\bigvee^{b} S^1)$. Consider the CW-structure on the torus $(S^1)^{b}$ given by identifying the opposite faces of the cube $[0,1]^{b}$, and let $(S^1)^{b}_d$ denote its $d$-skeleton.

\begin{theorem}[{\cite[Theorem 1.2]{Ong03}}]\label{Ong0312} For $d\leq {b}$, we have a homotopy equivalence
	\[\Sym^d\bigg(\bigvee^{b} S^1\bigg) \sim (S^1)^{b}_d.\]
\end{theorem}

\begin{remark}The proof is a generalization of Example \ref{exsim}, where $\mathbb{C}^*$ is replaced by $\mathbb{C}$ minus ${b}$ points in general position.
\end{remark}


\subsection{Lefschetz for $W_d(\Gamma)$} There is a natural inclusion $\Sym^d(\Gamma) \hookrightarrow \Div^d(\Gamma)$, realizing a point $p_1+\cdots+p_d$ of $\Sym^d(\Gamma)$ as an effective degree $d$ divisor on $\Gamma$. Thus, fixing a basepoint $p \in \Gamma$, we have a map $\alpha^{(d)}_p:\Sym^d(\Gamma) \rightarrow W_d(\Gamma)$ induced by the map $\alpha_{p,d}:\Div^d(\Gamma) \rightarrow W_d(\Gamma)$. The following is an essential ingredient in our proof of Theorem \ref{tropical wd}.

\begin{proposition}\label{trophomotopy} The map $\alpha^{(d)}_p:\Sym^d(\Gamma) \rightarrow W_d(\Gamma)$ is a homotopy equivalence.
\end{proposition}

\noindent{}Proposition \ref{trophomotopy} is a consequence of the contractibility of the fibers of $\alpha^{(d)}_p$, as we now explain. First, recall that given a divisor $D$, the \textbf{complete linear series} $|D|$ is the set of effective divisors equivalent to $D$. Each point $x$ of $W_d(\Gamma)$ corresponds to a class $[D] \in \Pic^d(\Gamma)$ containing an effective divisor $D$. This inclusion $\Sym^d(\Gamma) \hookrightarrow \Div^d(\Gamma)$ identifies $\Sym^d(\Gamma)$ with the subset of $\Div^d(\Gamma)$ of effective divisors, and therefore the preimage of $x$ in $\Sym^d(\Gamma)$ is equal to $|D|$. Consider $|D|$ as a topological space with the subspace topology from $\Sym^d(\Gamma)$. Then, by \cite[Corollary 31]{HMY12}, the complete linear series $|D|$ is contractible.


We also recall that a continuous function $\phi:M \rightarrow N$ is a \textbf{weak homotopy equivalence} if the induced map $\phi_{*,n}:\pi_n(M) \rightarrow \pi_n(N)$ is an isomorphism for all $n$. The following is a theorem of Smale, also known as the Vietoris-Begle-Smale mapping theorem.

\begin{theorem}[{\cite[Main Theorem]{Sma57}}]\label{Sma58MainThm}Let $\phi:M \rightarrow N$ be a proper surjection between connected locally compact metric spaces with a countable dense subset. Suppose $M$ is locally contractible, and suppose $\phi^{-1}(p)$ is contractible for all $p\in N$. Then $\phi$ is a weak homotopy equivalence.
\end{theorem}


\noindent{}By Whitehead's theorem, a weak homotopy equivalence $\phi:M \rightarrow N$ is a homotopy equivalence if both $M$ and $N$ have the homotopy type of CW-complexes. Given a CW-decomposition on $\Gamma$, there is a natural way of putting a CW-structure on $\Sym^d(\Gamma)$ (see \cite[\S 4K]{Hat02}). From \cite{LPP12}, we know that we can write $W_d(\Gamma)$ as the image of a finite union of polytopes in $H_1(\Gamma,\mathbb{R})$. In particular, this implies that $W_d(\Gamma)$ also admits a CW-decomposition.

\begin{proof}[Proof of Proposition \ref{trophomotopy}] By Whitehead's theorem and the Vietoris-Begle-Smale mapping theorem, it suffices to show that the map $\Sym^d(\Gamma) \rightarrow W_d(\Gamma)$ has contractible fibers. As observed above, the fibers of the map are contractible by \cite[Corollary 31]{HMY12}, and the theorem follows.
\end{proof}

\noindent{}We now prove Theorem \ref{tropical wd}.

\begin{proof}[Proof of Theorem \ref{tropical wd}] By \cite[Theorem 6.5]{MZ08}, every divisor of degree greater or equal to $b$ is equivalent to an effective divisor. In particular, for $d \geq b$, $W_d(\Gamma) = J(\Gamma)$. Therefore we only need to consider the case of $d< {b}$. By Proposition \ref{trophomotopy}, we can equivalently show that for $d< b$, the map $\Sym^d(\Gamma) \rightarrow J(\Gamma)$ induces isomorphisms between homotopy groups of dimensions $<d$ and a surjection in dimension $d$. 
	
Choose a basepoint $p \in \Gamma$. There is a natural map $\Gamma \rightarrow \Sym^d(\Gamma)$ given by
\[x\mapsto x+p+\cdots+p\]
such that the following diagram commutes.
\begin{center}
\begin{figure}[H]
			\includegraphics[scale=1]{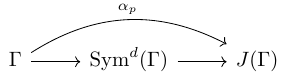}
\end{figure}
\end{center}
\vspace{-25pt}

By \cite[Lemma 3.4]{BR13}, the map $\alpha_{p,*}: H_1(\Gamma,\mathbb{Z}) \rightarrow H_1(J(\Gamma),\mathbb{Z})$ is an isomorphism. We now show that the map $H_1(\Gamma,\mathbb{Z}) \rightarrow H_1(\Sym^d(\Gamma),\mathbb{Z})$ is also an isormorphism. For $n>1$, we have an inclusion $\Sym^{n-1}(\Gamma) \hookrightarrow \Sym^{n}(\Gamma)$ given by 
\[x_1+\cdots +x_{n-1} \mapsto x_1+\cdots +x_{n-1}+p.\]
Since $\Gamma$ is connected, by \cite[Theorem 1.3]{Kal08}, the above inclusion induces isomorphisms between homology groups of dimensions $\leq 2n-3$. Therefore, the inclusions
\[\Gamma=\Sym^{1}(\Gamma) \hookrightarrow \cdots \hookrightarrow \Sym^{d}(\Gamma)\]
induce an isormophism $H_1(\Gamma,\mathbb{Z}) \cong H_1(\Sym^d(\Gamma),\mathbb{Z})$, as desired.
 
It thus follows that the map $H_1(\Sym^d(\Gamma),\mathbb{Z}) \rightarrow H_1(J(\Gamma),\mathbb{Z})$ is an isomorphism. Now, the Hurewicz maps $\pi_1(\Sym^d(\Gamma))\rightarrow H_1(\Sym^d(\Gamma,\mathbb{Z})$ and $\pi_1(J(\Gamma))\rightarrow H_1(J(\Gamma),\mathbb{Z})$ are isomorphisms. Hence, the map $\pi_1(\Sym^d(\Gamma)) \rightarrow \pi_1(J(\Gamma))$ is an isomorphism.

As noted above, $\Sym^d(\Gamma)$ is homotopy equivalent to $\Sym^d(\bigvee^{b} S^1)$. Thus, by Theorem \ref{Ong0312}, $\Sym^d(\Gamma)$ is homotopy equivalent to $(S^1)^{b}$. In particular, for $1< i<d$, we have that $\pi_i(\Sym^d(\Gamma)) \cong\pi_i((S^1)^{b}_d) =0 $. For $i>1$, $\pi_i(J(\Gamma))=0$. Hence, $\pi_i (\Sym^d(\Gamma)) \rightarrow \pi_i(J(\Gamma))$ is an isomorphism in dimensions $<d$, and a surjection in dimension $d$. 
\end{proof}

\begin{remark} In \cite{AB14}, Adiprasito and Bj\"orner proved, using a Morse-theoretic argument, that locally matroidal tropical varieties in $\mathbb{TP}^n$ satisfy a similar Lefschetz hyperplane theorem. The methods and results of \cite{AB14} do not apply to our setting, since  $W_d(\Gamma)$ is not necessarily locally matroidal. For example, if $\Gamma$ is the bouquet of circles from Example \ref{wd cube}, then $W_d(\Gamma)$ is locally isomorphic to the dual fan of a cube, which is not a matroidal fan. 
\end{remark}


\section{Analytification of Morphisms with Projective Fibers} \label{secanal}


\subsection{Berkovich analytification} Let $K$ be a field complete with respect to a non-trivial non-Archimedean norm $|\cdot|_K$. Berkovich analytification associates to each scheme $X$ that is locally of finite type over $K$ an analytic space $X^{an}$. If $X = \Spec A$, then the points of $X^{an}$ are multiplicative seminorms $|\cdot|:A\rightarrow \mathbb{R}_{\geq 0}$ extending the given norm on $K$. The topology on $X^{an}$ is the coarsest topology such that, for each $a \in A$, the function on $X^{an}$ given by $|\cdot|\mapsto |a|$ is continuous. In general, given an affine cover $X = \cup_i U_i$, there is a natural way to construct $X^{an}$ by gluing together the $U^{an}_i$: see \cite{Ber90} for more details.

An \textbf{extension} $L$ of $K$ is a field $K \subset L$ complete with respect to a norm $|\cdot|_L$ extending $|\cdot|_K$; $K$ is then a \textbf{non-Archimedean subfield} of $L$. Given an extension $L$ of $K$, there is a natural map $X(L) \rightarrow X^{an}$ defined as follow. Suppose we have a map $\Spec L \rightarrow X$. Let $U$ be an affine open in $X$ containing the image of $\Spec L$. Then, we obtain a norm $|\cdot|\in U^{an} \subset X^{an}$ by composing $|\cdot|_L$ with the map $\mathscr{O}_X(U) \rightarrow L$. Note that if $L'$ is an extension of $L$, then the following commutes, where the map $X(L) \rightarrow X(L')$ is given by viewing an $L$ point as an $L'$ point.

\begin{center}
\begin{figure}[H]
			\includegraphics[scale=1]{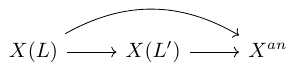}
\end{figure}
\end{center}
\vspace{-25pt}

Given a point $p\in X^{an}$, there is always an extension $L$ of $K$ such that $p$ is in the image of $X(L)$. Indeed, choose an affine open $U = \Spec A$ such that $p \in U^{an}$, then we can write $p$ as a seminorm $|\cdot|_p:A \rightarrow \mathbb{R}_{\geq 0}$. Let $\mathfrak{p}$ be the kernel of $|\cdot|_p$; then $\mathfrak{p}$ is prime. Let $k(p)$ denote the completion of the fraction field $k(\mathfrak{p})$ of $A_\mathfrak{p}$. The composition $\mathscr{O}_X(U) \rightarrow k(\mathfrak{p}) \rightarrow k(p)$ gives us a canonical preimage of $p$ in $X(k(p))$.

There is a nice correspondence between the scheme-theoretic properties of $X$ and the topological properties of $X^{an}$.
\begin{theorem}[{\cite[Theorem 3.4.8]{Ber90}}] The scheme $X$ is separated (resp. connected, resp. proper) if and only if $X^{an}$ is Hausdorff (resp. path-connected, resp. compact).
\end{theorem}

\noindent{}The sheaf of functions on $X$ determines a sheaf of analytic functions on $X^{an}$, and each morphism $f: X \rightarrow Y$ induces a morphism of ringed spaces $f^{an}: X^{an} \rightarrow Y^{an}$. Analytification is thus a functorial construction taking $X$ to the category $K\mbox{-}\mathbf{An}$ of $K$-analytic spaces. Since we are only interested in the topological properties of analytifications, we refer the reader to \cite{Ber90} for the details of these constructions. 

Suppose we have a morphism $f: X\rightarrow Y$. Then by \cite[Propostion 3.4.6]{Ber90}, the morphism $f$ is injective (resp. surjective) if and only if $f^{an}$ is injective (resp. surjective). Also, by \cite[Propostion 3.4.7]{Ber90}, if the morphism $f$ is of finite type, then $f$ is proper if and only if $f^{an}$ is proper. 

Let $p \in Y^{an}$. As discussed, there is an associated point $\Spec k(p) \rightarrow Y$ in $Y(k(p))$. Let $X_{k(p)}$ denote $X \times_Y \Spec k(p)$; the \textbf{analytic fiber} of $p$ is the analytification $X^{an}_{k(p)}$. Then, as noted in \cite[\S 1.4]{Ber93}, there is a natural homeomorphism
	\[X_{k(p)}^{an} \bij (f^{an})^{-1}(p).\]


\subsection{Morphisms with projective fibers} \label{subsectfiber} For the rest of this section, all schemes are assumed to be quasi-projective. Recall the embeddability result from Hrushovski, Loeser and Poonen.

\begin{theorem}[{\cite[Theorem 1.1]{HLP14}}]\label{HLP14T11} Let $X$ be a scheme over $K$ of dimension $d$. Then $X^{an}$ is homeomorphic to a subspace of $\mathbb{R}^{2d+1}$ if and only if $K$ has a countable dense subset. 
\end{theorem}

\noindent{}Also recall that using model theory techniques, Hrushovski and Loeser established in \cite{HL16} various topological tameness results for Berkovich analytifications (see \cite{Duc13} for an expository summary). In particular, they showed that Berkovich analytications are locally contractible and have the homotopy type of CW-complexes \cite[Theorem 14.4.1, Theorem 14.2.4]{HL16}. These results allow us to apply the Vietoris-Begle-Smale mapping theorem to derive the following.

\begin{lemma}\label{ber contra fiber} Suppose $K$ has a countable dense subset. Let $f: X\rightarrow Y$ be a surjective morphism of projective $K$-varieties such that for all $p\in Y^{an}$, the analytic fiber $X^{an}_{k(p)}$ is contractible. Then the map $f^{an}: X^{an}\rightarrow Y^{an}$ is a homotopy equivalence.
\end{lemma}

\begin{proof} By Theorem \ref{HLP14T11}, $X^{an}$ and $Y^{an}$ are locally compact metrizable, and have a countable dense subset. Since $X$ and $Y$ are projective, $f$ is proper and therefore $f^{an}$ is proper. Thus, by the Vietoris-Begle-Smale mapping theorem (Theorem \ref{Sma58MainThm}), $f^{an}$ is a weak homotopy equivalence. Since $X^{an}$ and $Y^{an}$ have the homotopy type of a CW-complex, $f^{an}$ is a homotopy equivalence by Whitehead's theorem.
\end{proof}

\noindent{}We no longer suppose that $K$ has a countable dense subset.

\begin{definition} A surjective morphism $f:X \rightarrow Y$ of $K$-schemes \textbf{satisfies property $(\dagger)$} if there is a finite stratification $Y = \coprod_i Y_i$, with $Y_i$ locally closed, such that $f:  X \times_Y Y_i \rightarrow Y_i$ is a projective bundle of rank $r_i$ over $Y_i$. 
\end{definition}
	
\noindent{}Let $K'$ be a subfield of $K$ and let $X'$ be a scheme over $K'$. For any extension $F$ of $K'$, let $X'_F$ denote $X' \otimes_{K'} F$.


\begin{lemma}\label{spreading out} Suppose we have a surjective morphism $f: X\rightarrow Y$ of projective $K$-varieties satisfying $(\dagger)$. Then, there exist a non-trivally valued non-Archimedean subfield $K' \subset K$, a model $X'$ (resp. $Y'$) over $K'$ of $X$ (resp. $Y$), and a morphism $f'_{K'} :X' \rightarrow Y'$ such that the following holds.
\begin{enumerate}
	\item The field $K'$ has a countable dense subset.
	\item The models $X'$ and $Y'$ are projective.
	\item The morphism $f'_{K'}$ is surjective and satisfies property $(\dagger)$.
	\item The morphism $f'_K = (f'_{K'} \otimes_{K'}K)$ is equal to $f$.
\end{enumerate}
\end{lemma}

\begin{proof} First, we show that there is a non-Archimedean subfield $Q$ of $K$ with a countable dense subset. Let $q = \chara(K)$, and let $\mathbb{F} = \mathbb{F}_q$ if $q$ is prime and $\mathbb{Q}$ otherwise. Then $\mathbb{F} \subset K$. Let $\mathfrak{T}$ be a non-zero element in the maximal ideal of the valuation ring $R$ of $K$, and let $Q$ be the completion of $\mathbb{F}(\mathfrak{T})$. Then $\mathbb{F}(\mathfrak{T})$ is a countable dense subset of $Q$.
	
Given a finite collection $\mathscr{T}:=\{T_1,...,T_n\}\subset K$, let $Q_{\mathscr{T}}$ be the completion of $Q(T_1,...,T_n)$. Then $Q_{\mathscr{T}}$ has a countable dense subset, i.e. the subfield $\mathbb{F}(\mathfrak{T},T_1,...,T_n)$. Since $K = \varinjlim Q_{\mathscr{T}}$, the lemma then follows from a spreading out argument (see \cite[Proposition 8.9.1 and Th\'eor\`eme 8.10.5]{Gro66}).
\end{proof}

{\renewcommand{\thetheorem}{\ref{projective fibers}} 
\begin{theorem} Suppose we have a surjective morphism $f: X\rightarrow Y$ of projective $K$-varieties satisfying $(\dagger)$. Then there is a finite extension $K\subset L$ such that $f^{an}_L:(X_L)^{an} \rightarrow (Y_L)^{an}$ is a homotopy equivalence.
	
Moreover, if we suppose the field $K$ has a countable dense subset, then we can take $L = K$.
\end{theorem} 
\addtocounter{theorem}{-1}}

\begin{proof} First, suppose $K$ has a countable dense subset. Then, given $p \in Y^{an}_i \subset Y^{an}$, let $$X^i_{k(p)}:=X_i \times_{Y_i} \Spec k(p).$$
By $(\dagger)$, we have $X_{k(p)} = X^i_{k(p)} \cong \mathbb{P}^{r_i}_{k(p)}$, and thus the fiber $(f^{an})^{-1}(p)$ is homeomorphic to $(\mathbb{P}^{r_i}_{k(p)})^{an}$, which is contractible. By Lemma \ref{ber contra fiber}, $f^{an}$ is a homotopy equivalence.

In general, we choose $K'$, $X'$, $Y'$ and $f'_{K'}$ as in Lemma \ref{spreading out}. By \cite[Theorem 14.2.3]{HL16}, there exists a finite extension $L'$ of $K'$ such that for all extensions $F$ of $L'$, the maps $(X'_{F})^{an} \rightarrow (X'_{L'})^{an}$ and $(Y'_{F})^{an} \rightarrow (Y'_{L'})^{an}$ are homotopy equivalences. Then, $f'_{L'}$ also satisfies $(\dagger)$, and thus $(f'_{L'})^{an}$ is a homotopy equivalence since $L'$, being finite over $K'$, has a countable dense subset.

Let $L$ be a finite extension of $K$ containing $L'$. Then $(X'_{L})^{an} \rightarrow (X_{L'})^{an}$ and $(Y'_{L})^{an} \rightarrow (Y_{L'})^{an}$ are homotopy equivalences by the choice of $L$. Hence $f^{an}_L = f^{an}_{L'} \otimes_{L'} L$ is a homotopy equivalence.
 \end{proof}


\section{Skeletons and Product of Strictly Polystable Models}\label{sec skel}

In this section, we assume the residue field $k$ of $K$ to be algebraically closed. Given $X$ over $K$, we say that a CW-complex $\Delta$ is a \textbf{skeleton} of $X^{an}$ if there is an inclusion $\Delta \hookrightarrow X^{an}$ and a deformation retraction $h_t : X^{an} \rightarrow X^{an}$ onto the image of $\Delta$. Let $R$ denote the valuation ring of $K$.
In this section, we review the construction from \cite{Ber99} of skeletons of $X^{an}$ given by strictly polystable $R$-models $\mathcal{X}$ of $X$.

\subsection{Skeletons of Strictly Polystable Models} From this point onward, fix a non-zero element $\mathfrak{T}$ in the maximal ideal of $R$.  A locally finitely presented formal scheme $\mathfrak{X}$ over $R$ is \textbf{strictly polystable}  if, for every $x\in \mathfrak{X}$, there is an affine neighborhood $\mathfrak{U}$ of $x$ such that the morphism $\mathfrak{U} \rightarrow \Spf R$ factors through an \'etale morphism $\mathfrak{U} \rightarrow \Spf B_0 \times..\times \Spf B_j$ where each $B_i$ is of the form $$R\{T_0,...,T_n\}/(T_0\cdot...\cdot T_n -a)$$ for some $a\in R$. A scheme $\mathcal{X}$ over $R$ is \textbf{strictly polystable} if its $\mathfrak{T}$-adic completion $\mathfrak{X}$ is strictly polystable.

Given $\mathfrak{X}$ strictly polystable, let $\mathfrak{X}_\eta$ denote the generic fiber of $\mathfrak{X}$ (in the category $K\mbox{-}\mathbf{An}$). In the paper mentioned above, Berkovich constructed a CW-complex $\Delta(\mathfrak{X})$ associated to $\mathfrak{X}$, called the \textbf{skeleton of $\mathfrak{X}$}. The complex $\Delta(\mathfrak{X})$ has a natural inclusion $\iota^{\mathfrak{X}}:\Delta(\mathfrak{X}) \hookrightarrow \mathfrak{X}_\eta$, and a there is an associated deformation retraction $h^{\mathfrak{X}}_t: \mathfrak{X}_\eta \rightarrow \mathfrak{X}_\eta$ onto the image of $\Delta(\mathfrak{X})$.

\begin{remark} Suppose we have a strictly polystable $R$-model $\mathcal{X}$ of a scheme $X$ over $K$, and suppose $\mathcal{X}$ is proper. Then, the generic fiber $\mathfrak{X}_\eta$ of the $\mathfrak{T}$-adic completion of $\mathcal{X}$ is in fact equal to $X^{an}$. Thus, $\Delta(\mathfrak{X})$ is a skeleton of $X^{an}$.
\end{remark}

\begin{example}\label{semistable} A scheme $\mathcal{X}$ over $R$ is called \textbf{strictly semistable} if for all $x\in \mathcal{X}$ there is a neighborhood $\mathcal{U}$ of $x$ such that the map $\mathcal{U} \rightarrow \Spec R$ factors through an \'etale morphism $\mathcal{U}\rightarrow \Spec B$ where $B$ is of the form $$R[T_0,...,T_{n+k}]/(T_0\cdot...\cdot T_n-a)$$ for some non-zero element $a$ in the maximal ideal of $R$.
	
Clearly being strictly semistable implies being strictly polystable. Let $\mathcal{X}_0$ denote the special fiber of $\mathcal{X}$. Then, $\mathcal{X}_0$ is a simple normal crossing divisor. Let $\{Z_i:i\in I\}$ be the set of irreducible components of $\mathcal{X}_0$. Let $\mathcal{P}(\mathcal{X}_0)$ be the poset with underlying set
\[\{W \subset \mathcal{X}_0: W\mathrm{\ is\ a\ irreducible\ component\ of} \cap_{i\in J} Z_i\mathrm{\ for\ any\ } J\subset I\}\]
and ordering given by reverse inclusion. Recall that the \textbf{dual complex} $\Delta(\mathcal{X}_0)$ of $\mathcal{X}_0$ is a $\Delta$-complex whose poset of faces is naturally isomorphic to $\mathcal{P}(\mathcal{X}_0)$. Let $\mathfrak{X}$ denote the $\mathfrak{T}$-adic completion of $\mathcal{X}$, then there is a canonical isomorphism $\Delta(\mathfrak{X}) \cong \Delta(\mathcal{X}_0)$.
 	
\end{example}


\subsection{Skeletons of Products} 

Let $\mathfrak{X}$ and $\mathfrak{Y}$ be two strictly polystable formal schemes over $R$. Since we assumed $K$ to be algebraically closed, by \cite[Lemma 3.16]{Ber99}, there is a canonical homeomorphism between the skeleton $\Delta(\mathfrak{X} \times \mathfrak{Y})$ of the fiber product $\mathfrak{X} \times \mathfrak{Y}$ and the product $\Delta(\mathfrak{X}) \times \Delta(\mathfrak{Y})$ of the skeletons of $\mathfrak{X}$ and $\mathfrak{Y}$. 

In particular, let $\mathfrak{X}^d$ denote the $d$-fold fiber product over $\Spf R$ of a strictly polystable formal scheme $\mathfrak{X}$. Then $\mathfrak{X}^d$ is strictly polystable, and there is a canonical homeomorphism $\Delta(\mathfrak{X}^d) \cong \Delta(\mathfrak{X})^d$ with the following properties.

\begin{proposition} \label{projection product}  Let $\pi_i^{\Delta}$ denote the projection of $\Delta(\mathfrak{X})^d$ onto its $i$-th factor. Then, by identifying $\Delta(\mathfrak{X})^d$ with $\Delta(\mathfrak{X}^d)$, we obtain the following commuting diagram where the left vertical arrow is given by the inclusion $\iota^{\mathfrak{X}^d}: \Delta(\mathfrak{X}^d) \hookrightarrow \mathfrak{X}^d_\eta$.
\begin{figure}[H]\label{prodinclusion}
			\includegraphics[scale=1]{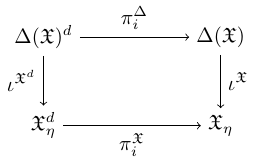}
\end{figure}
\noindent{}Moreover, let $r^{\mathfrak{X}}=h^{\mathfrak{X}}_1$ denote the retraction map from $\mathfrak{X}_\eta$ onto $\Delta(\mathfrak{X})$, and let $r^{\mathfrak{X}^d}=h^{\mathfrak{X}^d}_1$ denote the retraction map from $\mathfrak{X}^d_\eta$ onto $\Delta(\mathfrak{X}^d)\cong\Delta(\mathfrak{X})^d$. 

\begin{figure}[H]\label{prodretraction}
			\includegraphics[scale=1]{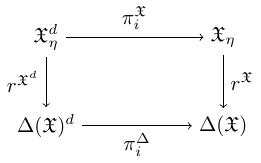}
\end{figure}
\end{proposition}
\begin{proof} By \cite[Theorem 5.2.vii]{Ber99}, it suffices to show that the first diagram commutes. In other words, we need to show that the restriction $\pi^{\mathfrak{X}}_i |_{\Delta(\mathfrak{X})^d}$ is equal to $\pi_i^{\Delta}$, which follows from \cite[Theorem 5.4]{Ber99} and the functoriality of the identification $\Delta(\mathfrak{X}^d) \cong \Delta(\mathfrak{X})^d$ (see \cite[Lemma 3.16]{Ber99}). \end{proof}


\section{Skeletons of Symmetric Products} \label{sec sym} 
Again, we assume in this section that the residue field $k$ of $K$ is algebraically closed. Let $X$ denote a projective $K$-scheme with a proper strictly polystable $R$-model $\mathcal{X}$. For $d\geq 1$, the $d$-th symmetric group $S_d$ acts on $X^d$ by permuting the factors. The \textbf{$d$-th symmetric product} $\Sym^d(X)$ of $X$ is the scheme-theoretic quotient $X^d\sslash S_d$ in $K\mbox{-}\mathbf{Sch}$. Similarly, $S_d$ acts on the product $\mathcal{X}^d$ over $R$ by permuting the factors. The \textbf{$d$-th relative symmetric product} $\Sym^d(\mathcal{X})$ is the quotient $\mathcal{X}^d\sslash S_d$ in $R\mbox{-}\mathbf{Sch}$.

Let $\Sym^{d,an}(X)$ denote the analytification of $\Sym^d(X)$, and let $\mathfrak{Sym}^d$ denote the $\mathfrak{T}$-adic completion of $\Sym^d(\mathcal{X})$. Since $\Sym^d(\mathcal{X})$ is proper and is an $R$-model of $\Sym^d(X)$, one could hope to use $\mathfrak{Sym}^d$ to construct a skeleton for $\Sym^{d,an}(X) = \mathfrak{Sym}^d_\eta$. However, aside from a few special cases, such as when $X$ is a smooth curve with good reduction, $\Sym^d(\mathcal{X})$ is highly singular and far from being strictly polystable. 

Therefore, instead of working with $\mathfrak{Sym}^d$, we proceed by establishing the following correspondence theorem, which will allow us to realize the complex $\Sym^d(\Delta(\mathfrak{X}))$ as a skeleton of $\Sym^{d,an}(X)$. Let $G$ be any finite group acting on a projective $K$-variety $Y$. Recall that the action of $G$ on $Y$ induces an action of $G$ on the topological space $Y^{an}$; each $\sigma \in G$ can be viewed as an automorphism $\sigma: Y \rightarrow Y$, and the analytification $\sigma^{an}:Y^{an} \rightarrow Y^{an}$ is a homeomorphism.  Let $Y\sslash G$ denote the quotient in $K\mbox{-}\mathbf{Sch}$ of $Y$ by $G$, and let $Y^{an}/G$ denote the quotient in $\mathbf{Top}$ of $Y$ by $G$.

\begin{theorem}\label{analquotient} Let $\pi_G: Y \rightarrow Y\sslash G$ denote the quotient map, and let $\pi_G^{an}$ denote its analytification. Then there is a canonical homeomorphism $(Y\sslash G)^{an} \cong Y^{an}/G$, identifying $\pi_G^{an}$ with the quotient $Y^{an} \rightarrow Y^{an}/G$.
\end{theorem}

\subsection{Quotient by finite group actions}
We establish the above theorem by reducing it to the affine case. Let $A$ be a ring of finite type over $K$ with an action by a finite group $G$. Let $U$ denote $\Spec A$. Then, the action of $G$ on $A$ induces an action of $G$ on $U$, which in turn induces an action of $G$ on the analytificantion $U^{an}$.

Recall from \cite[Proposition V.1.1]{Gro71} that the quotient scheme $U\sslash G$ is equal to $\Spec A^G$ with $A^G$ denoting the $G$-invariant subring. The natural map $\pi_G: U \rightarrow U\sslash G$ is finite and $G$-equivariant, implying that the induced map $\pi^{an}_G: U^{an} \rightarrow (U/G)^{an}$ is finite and $G$-equivariant.

\begin{lemma}\label{lemmaquot} As a continuous map between topological spaces, the map $\pi^{an}_G$ is a quotient map, i.e. it is surjective and sends open sets to open sets. Moreover, $G$ acts transitively on the fibers of $\pi^{an}_G$.
\end{lemma}

\begin{proof} Surjectivity of $\pi^{an}_G$ follows from surjectivity of $\pi_G$. 
Recall now that the topology on $U^{an}$ is the subspace topology for the natural inclusion
\[U^{an} \lhook\joinrel\longrightarrow (\mathbb{R}_{\geq 0})^{A}.\]
Since $(U/G)^{an}$ is equal to $(\Spec A^G)^{an}$, we have the following diagram of continuous maps,
\begin{center}
\begin{figure}[H]
			\includegraphics[scale=1]{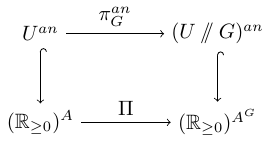}
\end{figure}
\end{center}
\vspace{-25pt}
where the map $\Pi$ is defined by sending $f \in (\mathbb{R}_{\geq 0})^{A}$ to $f |_{A^G}$.
For $a\in A$, let $\mathrm{ev}_{a,A}$ denote the evalution map $\mathrm{ev}_{a,A}:(\mathbb{R}_{\geq 0})^{A}\rightarrow \mathbb{R}_{\geq 0}$ defined by sending $f \in (\mathbb{R}_{\geq 0})^{A}$ to $f(a)$. A subbasis of the topology on $(\mathbb{R}_{\geq 0})^A$ is given by the set
\[\{\mathrm{ev}_{a,A}^{-1}(\Omega): a \in A, \Omega \mathrm{\ open\ in\ }\mathbb{R}_{\geq 0}\}.\]
Similarly, for $a \in A^G$, let $\mathrm{ev}_{a,A^G}$ denote the evalution map $\mathrm{ev}_{a,A^G}:(\mathbb{R}_{\geq 0})^{A^G}\rightarrow \mathbb{R}_{\geq 0}$ defined by sending $f \in (\mathbb{R}_{\geq 0})^{A^G}$ to $f(a)$. A subbasis of the topology on $(\mathbb{R}_{\geq 0})^{A^G}$ is given by the set
\[\{\mathrm{ev}_{a,A^G}^{-1}(\Omega): a \in A^G, \Omega \mathrm{\ open\ in\ }\mathbb{R}_{\geq 0}\}.\]
Then, given $a\in A$ and $\Omega$ open in $\mathbb{R}_{\geq 0}$, if $a$ is not contained in $A^G$ we have that $\Pi(\mathrm{ev}_{a,A}^{-1}(\Omega))$ is equal to $(\mathbb{R}_{\geq 0})^{A^G}$, which is open in $(\mathbb{R}_{\geq 0})^{A^G}$. If $a$ is contained in $A^G$, then $\Pi(\mathrm{ev}_{a,A}^{-1}(\Omega))$ is equal to $\mathrm{ev}_{a,A^G}^{-1}(\Omega)$, which is also open in $(\mathbb{R}_{\geq 0})^{A^G}$. Therefore, $\Pi$ sends open sets to open sets, which implies that $\pi^{an}_G$ sends open sets to open sets.

Finally, we need to verify that $G$ acts transitively on the fibers of $\pi^{an}_G$. This is established in Step 1 of the proof of \cite[Theorem 3.1]{Han16} (in the context of adic spaces). We summarize the argument here for completeness. Let $p$ denote a norm $|\cdot|_{p}: A^{G} \rightarrow \mathbb{R}_{\geq 0}$ in $(U/G)^{an}$. Let $\mathfrak{p}$ denote the kernel of $p$, then we can think of $p$ as a valuation of the fraction field $\mathfrak{F_p}=\mathrm{Frac}(A^G/\mathfrak{p})$.

By the going-up theorem, we can choose a prime ideal $\mathfrak{q}$ such that $\mathfrak{q} \cap A^G = \mathfrak{p}$. Let $G_\mathfrak{q}$ denote the stabilizer of $\mathfrak{q}$, and $L$ denote the fraction field $L=\mathrm{Frac}(A/\mathfrak{q})$. Then, $G$ acts transitively on primes $\mathfrak{q}$ such that $\mathfrak{q} \cap A^G = \mathfrak{p}$, $L$ is a normal algebraic extension of $\mathfrak{F_p}$, and the induced map $G_\mathfrak{q} \rightarrow \mathrm{Aut}(L/\mathfrak{F_p})$ is surjective (see \cite[Proposition V.1.1]{Gro71}). By \cite[Corollary VI.7.3]{SZ60}, $\mathrm{Aut}(L/\mathfrak{F_p})$ acts transitively on valuations of $L$ extending $p$, which implies that $G_{\mathfrak{q}}$ acts transitively on such extensions of $p$.\end{proof} 

\noindent{}Let $Y$ be a projective variety over $K$ with an action by a finite group $G$. Let $Y\sslash G$ denote the scheme-theoretic quotient of $X$ by $G$. The action of $G$ on $Y$ induces an action of $G$ on $Y^{an}$, and as noted above, the maps $\pi_G: X \rightarrow X\sslash G$ and $\pi_G^{an}: X^{an} \rightarrow (X\sslash G)^{an}$ are $G$-equivariant. 

By considering $Y^{an}$ as a topological space, we can take the topological quotient $Y^{an}/ G$. Let $\kappa_G:Y^{an}\rightarrow Y^{an}/ G$ denote the associated quotient map. Since $\pi_G^{an}$ is $G$-equivariant, we have an induced continuous map 
\[\iota_G: Y^{an}/ G \rightarrow (Y\sslash G)^{an}\]
such that $\iota_G \circ \kappa_G = \pi_G^{an}$. To establish Theorem \ref{analquotient}, it suffices to show that $\iota_G$ is a homeomorphism.
\begin{proof}[Proof of Theorem \ref{analquotient}] We have the following diagram.

\begin{center}
\begin{figure}[H]
			\includegraphics[scale=1]{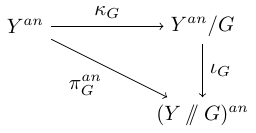}
\end{figure}
\end{center}
\vspace{-25pt}
Since $\pi_G^{an}$ is surjective, so is $\iota_G$. Let $\Omega$ denote an open set in $Y^{an}/ G$. Choose an affine open cover $\{U^G_i\}$ of $Y\sslash G$. Since $\pi_G$ is finite, each preimage $U_i = \pi^{-1}_G(U^G_i)$ is a $G$-invariant affine open subscheme of $Y$ with $U^G_i = U_i \sslash G$. Since $\kappa_G$ is continuous, $\kappa_G^{-1}(\Omega)$ is open. By Lemma \ref{lemmaquot}, each restriction $\pi_G^{an}|_{U_i^{an}}$ is a quotient map and therefore sends each $\kappa_G^{-1}(\Omega) \cap U_i^{an}$ to an open subset $\Omega_i^G$ of $(U^G_i)^{an}$. It follows from the equality
\[\pi_G^{an}(\kappa_G^{-1}(\Omega)) = \bigcup_i \Omega_i^G\]
that $\pi_G^{an}(\kappa_G^{-1}(\Omega))$ is open. Since $\iota_G(\Omega) =\pi_G^{an} (\kappa^{-1}(\Omega))$, we obtain that $\iota_G$ is a quotient map.

To conclude the proof, we need to show that $\iota_G$ is injective. Let $z$ be a point in $(Y \sslash G)^{an}$, and suppose we have $y$ and $y'$ in $Y^{an}$ such that $\pi_G^{an}(y) = \pi_G^{an}(y')=z$. Choose an affine open $U^G \subset Y\sslash G$ such that $z\in (U^G)^{an}$. Let $U$ denote the preimage $\pi_G^{-1}(U^G)$. Then, $y$ and $y'$ are contained in $U^{an}$. Since $U^G = U\sslash G$, by Lemma \ref{lemmaquot} we have that $G$ acts transitively on the fibers of $\pi^{an}_G|_{U^{an}}:U^{an} \rightarrow (U^G)^{an}$. In particular, there is an element $\sigma \in G$ such that $\sigma \cdot y = y'$, which implies that $\kappa_G(y) = \kappa_G(y')$. Finally, since $\kappa_G$ is surjective and the choices of $y$ and $y'$ were arbitrary, $\iota_G$ is injective.
\end{proof}


\subsection{The quotient of the skeleton} 
Let $X$ be a projective $K$-scheme with a proper strictly polystable $R$-model $\mathcal{X}$. Let $X^{d,an}$ denote the analytification of $X^d$. Since the $\mathfrak{T}$-adic completion $\mathfrak{X}^d$ of $\mathcal{X}^d$ is strictly polystable, Proposition \ref{projection product} realizes the complex $\Delta(\mathfrak{X})^d$ as a skeleton of $X^{d,an} = \mathfrak{X}^d_\eta$. By invoking Theorem \ref{analquotient}, we can now avoid working with $\mathfrak{Sym}^d$, and instead construct a skeleton of $\Sym^{d,an}(X)$ by showing that the deformation retraction $h^{\mathfrak{X}^d}_t$ from $X^{d,an}$ onto $\Delta(\mathfrak{X})^d$ is $S_d$-invariant.

Choose $\sigma \in S_d$. The corresponding automorphism $\sigma: \mathcal{X}^d \rightarrow \mathcal{X}^d$ induces, via $\mathfrak{T}$-adic completion, an automorphism $\sigma^{\mathfrak{X}}: \mathfrak{X}^d \bij \mathfrak{X}^d$, which in turn induces an automorphism $\sigma^{\mathfrak{X}}_\eta: \mathfrak{X}^d_\eta \bij \mathfrak{X}^d_\eta$. Since $\mathfrak{X}^d_\eta= X^{d,an}$, this defines an action of $S_d$ on $X^{d,an}$, which agrees with the action on $X^{d,an}$ discussed previously.

By \cite[Theorem 5.2.vii]{Ber99}, $\sigma^{\mathfrak{X}}_\eta: \mathfrak{X}^d_\eta \bij \mathfrak{X}^d_\eta$ restricts itself to a homeomorphism $\sigma^{\Delta}: \Delta(\mathfrak{X})^d \bij \Delta(\mathfrak{X})^d$. This defines an action of $S_d$ on $\Delta(\mathfrak{X})^d$.

\begin{lemma}\label{permuting factors} The above action of $S_d$ on $\Delta(\mathfrak{X})^d$ is precisely the action of $S_d$ on $\Delta(\mathfrak{X})^d$ given by permuting the factors. In particular, $\Delta(\mathfrak{X})^d/S_d = \Sym^d(\Delta(\mathfrak{X}))$.
\end{lemma}

\begin{proof} Given $\sigma \in S_d$ and $1 \leq i \leq d$, the composition $\pi^{\mathfrak{X}}_i \circ \sigma^{\mathfrak{X}}_\eta : \mathfrak{X}^d_\eta \rightarrow \mathfrak{X}_\eta$ is equal to $\pi^\mathfrak{X}_{\sigma(i)}$. By Proposition \ref{projection product}, $\pi^\Delta_i \circ \sigma^\Delta$ is equal to the restriction $\pi^{\mathfrak{X}}_i |_{\Delta(\mathfrak{X})^d} \circ \sigma^{\mathfrak{X}}_\eta |_{\Delta(\mathfrak{X})^d} = (\pi^{\mathfrak{X}}_i \circ \sigma^{\mathfrak{X}}_\eta)|_{\Delta(\mathfrak{X})^d}$. Therefore $\pi^\Delta_i \circ \sigma^\Delta = \pi^\Delta_{\sigma(i)}$.
	
Let $p = (p_1,...,p_d)\in \Delta(\mathfrak{X})^d$. Then we have \[\pi^\Delta_i(\sigma^\Delta(p))= \pi^\Delta_{\sigma(i)}(p) = p_{\sigma(i)}.\]
Hence $\sigma^\Delta (p_1,...,p_d) = (p_{\sigma(i)},...,p_{\sigma(d)})$, which is the desired statement.
\end{proof}

\noindent{}By Theorem \ref{analquotient} and Lemma \ref{permuting factors}, the natural inclusion $\Delta(\mathfrak{X})^d/S_d \hookrightarrow \mathfrak{X}^d_\eta / S_d$ defines an inclusion $\Sym^d(\Delta(\mathfrak{X})) \hookrightarrow \Sym^{d,an}(X)$, which leads us to the following.

\begin{theorem}\label{symskel}There is a natural inclusion $\iota^{S_d}: \Sym^d(\Delta(\mathfrak{X})) \hookrightarrow \Sym^{d,an}(X)$ and a deformation retraction $h^{S_d}_t:\Sym^{d,an}(X) \rightarrow \Sym^{d,an}(X)$ onto the image of $\Sym^d(\Delta(\mathfrak{X}))$. Moreover, let $$r^{S_d} =h^{S_d}_1: \Sym^{d,an}(X) \rightarrow \Sym^d(\Delta(\mathfrak{X}))$$ denote the retraction map. Then the following diagram commutes
\begin{figure}[H]\label{diagramskeletons}
			\includegraphics[scale=1]{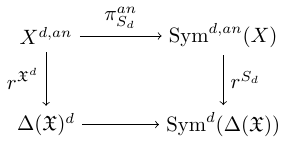}
\end{figure}
\noindent{}where $\Delta(\mathfrak{X})^d \rightarrow \Sym^d(\Delta(\mathfrak{X}))$ is the natural quotient map.
\end{theorem}

\begin{proof}The inclusion $\iota^{S_d}$ has already been constructed. Now, consider the deformation retraction $h^{\mathfrak{X}^d}_t: \mathfrak{X}^d \rightarrow \mathfrak{X}^d$ onto the image of $\Delta(\mathfrak{X})^d$. By \cite[Theorem 5.2.vii]{Ber99}, for all $\sigma \in S_d$ and $t \in [0,1]$, we have $\sigma^\mathfrak{X}_\eta \circ h^{\mathfrak{X}^d}_t = h^{\mathfrak{X}^d}_t \circ \sigma^\mathfrak{X}_\eta$. Therefore, we obtain a family of functions $$h^{S_d}_t:= h^{\mathfrak{X}^d}_t/S_d : \mathfrak{X}^d_\eta/S_d \longrightarrow \mathfrak{X}^d_\eta/S_d$$
which gives deformation retraction from $\Sym^{d,an}(X)$ onto the image of $\Sym^d(\Delta(\mathfrak{X}))$. The commutativity of the diagram is immediate from the construction of $h_t^{S_d}$.
\end{proof}

\section{Tropicalizing the Abel-Jacobi map}\label{sectionabeljac}

Let $X$ be a smooth projective curve over $K$ of genus $g$. Assume that $X$ has a strictly semistable $R$-model $\mathcal{X}$; let $\mathfrak{X}$ denote the $\mathfrak{T}$-adic completion of $\mathcal{X}$. As discussed in Example \ref{semistable}, there is a canonical identification between the skeleton $\Delta(\mathfrak{X})$ and the dual graph $G(\mathcal{X}_0)$ of the special fiber $\mathcal{X}_0$ of $\mathcal{X}$.

Each edge $e$ in $G(\mathcal{X}_0)$ corresponds uniquely to a node $x_e$ of $\mathcal{X}_0$, and each $x_e$ has a neighborhood $\mathcal{U}_e$ admitting an \'etale morphism
\[\mathcal{U}_e \rightarrow \Spec R[x,y]/(xy-a_e)\]
for some non-zero $a_e$ in the maximal ideal of $R$. Set $\ell_\mathcal{X}(e) := - \log|a_e|_K$. The pair $(G(\mathcal{X}),\ell_\mathcal{X})$ defines a metric graph $\Gamma = \Gamma(\mathcal{X})$, namely the \textbf{tropicalization} of $X$ with respect to $\mathcal{X}$. Since $\Gamma$ is naturally homeomorphic to $\Delta(\mathfrak{X})$, we have a tropicalization map $X^{an} \rightarrow \Gamma$ given by the retraction $r^{\mathfrak{X}}$ of $X^{an}$ onto $\Delta(\mathfrak{X})$. For details and references on skeletons and tropicalizations of non-Archimedean curves, see \cite{Ber90,BPR13}.


\subsection{Retraction of divisors} Assume, for the rest of the section, that the residue field $k$ of $R$ is algebraically closed. Let $L$ be an extension of $K$, and let $B \subset L$ be its valuation ring. Then $\mathcal{X}_B:=\mathcal{X} \otimes_R B$ is a strictly semistable $B$-model of $X_L$, and there is an canonical isomorphism between $\Gamma$ and the tropicalization of $X_L$ with respect to $\mathcal{X}_B$. In particular, we have a map $r_L: X_L(L) \rightarrow \Gamma$, given by composing the tropicalization map from $X^{an}_L$ onto $\Gamma$ with the natural map $X_L(L) \rightarrow X^{an}_L$. Extending $r_L$ by linearity gives maps $$r^d_L:\Div^d_L(X_L)\rightarrow \Div^d(\Gamma),$$ where $\Div^d_L(X_L)$ is the set of degree $d$ divisors on $X_L$ supported on $X(L)$.


Suppose now that $L$ is algebraically closed, then $\Div^d_L(X_L) = \Div^d(X_L)$. Since $X$ is a smooth projective curve, we have an inclusion $\Sym^d(X)(L) \hookrightarrow \Div^d(X_L)$, realizing $\Sym^d(X)(L)$ as the subset of degree $d$ effective divisors.

\begin{proposition}\label{symdiv} The following diagram commutes, where $r^{S_d}_L$ is given by composing the retraction map $r^{S_d}: \Sym^{d,an}(X) \rightarrow \Sym^d(\Gamma)$ with the natural map $\Sym^d(X)(L) \rightarrow \Sym^{d,an}(X)$.
\begin{figure}[H]
			\includegraphics[scale=1]{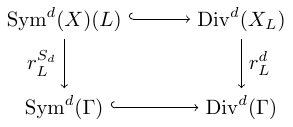}
\end{figure}
\end{proposition}

\begin{proof}Consider the following diagram, with the left square being induced by the commutative diagram from Theorem \ref{symskel}.
\begin{figure}[H]
				\includegraphics[scale=1]{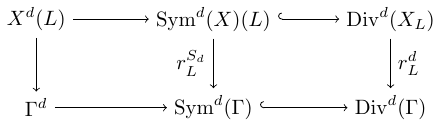}
\end{figure}
\noindent{}To show that the right square commutes, it suffices to show that the outer square commutes. For $1\leq i \leq d$, let $\pi^L_i:X^d(L)\rightarrow X(L)$ denote the projection onto the $i$th factor. Given $P \in X^d(L)$, let $P_i := \pi^L_i(P)$, and let $p_i:=r_L(P_i)$. Then $X^d(L) \rightarrow \Gamma^d$ sends $P$ to $(p_1,...,p_d)$, and $X^d(L) \rightarrow \Div^d(X_L)$ sends $P$ to $P_1+\cdots+P_d$. Hence the outer square commutes, as both $X^d(L) \rightarrow \Div^d(X_L) \rightarrow \Div^d(\Gamma)$ and $X^d(L) \rightarrow \Gamma^d \rightarrow \Div^d(\Gamma)$ takes $P$ to $p_1+\cdots+p_d$.
\end{proof}


\subsection{Abel-Jacobi} Let $J$ denote the Jacobian of $X$. Recall from \cite[Theorem 1.3]{BR13} that there is a natural inclusion $\iota^J$ from the Jacobian torus $J(\Gamma)$ into $J^{an}$, and a deformation retraction $h^J_t: J^{an} \rightarrow J^{an}$ onto the image of $J(\Gamma)$. 

Suppose that $X(K) \neq \emptyset$. Fix a basepoint $P\in X(K)$, and let $p:= r_K(P)$. Let $\alpha_P: X \rightarrow J$ (resp. $\alpha_p: \Gamma \rightarrow J(\Gamma)$) denote the Abel-Jacobi map based at $P$ (resp. the tropical Abel-Jacobi map based at $p$). Recall from \cite[Proposition 6.1]{BR13} that the following diagram commutes, where $r^J$ is the retraction map $J^{an} \rightarrow J(\Gamma)$.
\begin{figure}[H]
	\includegraphics[scale=1]{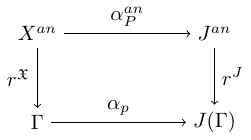}
\end{figure}

\noindent{}The map $\alpha_P$ induces an isomorphism $\widetilde{\alpha}_{P,d}: \Pic^d(X) \bij J$. Let $\alpha^{(d)}_P: \Sym^d(X) \rightarrow J$ denote the map given by composing $\widetilde{\alpha}_{P,d}$ with the natural map $\Sym^d(X) \rightarrow \Pic^d(X)$. We now establish the following generalization of \cite[Proposition 6.1]{BR13}.

\begin{proposition}\label{symabeljac}
Let $\alpha^{(d)}_p: \Sym^d(\Gamma) \rightarrow J(\Gamma)$ be the map  $p_1+\cdots+p_d \mapsto \alpha_p(p_1)+\cdots+\alpha_p(p_d)$. Then the following diagram commutes.
\begin{figure}[H]
	\includegraphics[scale=1]{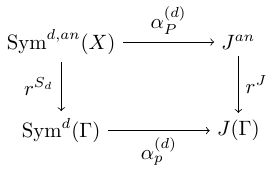}
\end{figure}
\end{proposition}
\begin{proof} Let $L$ be an extension of $K$. Without loss of generality, we replace $L$ by its algebraic closure. Consider the following diagram.
\begin{figure}[H]
	\includegraphics[scale=1]{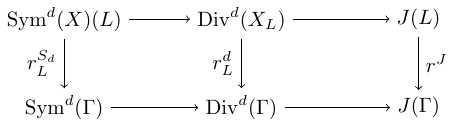}
\end{figure}
\noindent{}The left square commutes by Proposition \ref{symdiv}; for the right square, see the proof of \cite[Proposition 6.1]{BR13}. Therefore the outer and top squares of the following commute for all $L$.
\begin{figure}[H]
	\includegraphics[scale=1]{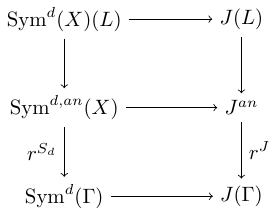}
\end{figure}
\noindent{}The commutativity of the bottom square now follows, as every points $x \in \Sym^{d,an}(X)$ lies in the image of $\Sym^d(X)(L)$ for some $L$.
\end{proof}


\section{Non-Archimedean Lefschetz}\label{finalsec} We continue to follow the notations of \S \ref{sectionabeljac}. Throughout this section, $K$ is assumed to be algebraically closed. In particular, $X(K)\neq \emptyset$, the residue field $k$ is algebraically closed, and $X$ always admits a strictly semistable $R$-model $\mathcal{X}$. Fix $\mathcal{X}$.


\subsection{Tropicalizing effective divisors} Fix a basepoint $P\in X(K)$, and let $p:= r_K(P)$. Let $W_d \subset J$ denote the image of the map $\alpha^{(d)}_P$. 

Given $r \geq 0$, let $W^r_d \subset \Pic^d(X)$ denote the locus of divisor classes of degree $d$ and of rank at least $r$, which is a closed subscheme of $\Pic^d(X)$. Recall that for $r=0$, we have an identification of $W^0_d$ with $W_d$ given by the bijection $\widetilde{\alpha}_{P,d}: \Pic^d(X) \bij J$. For more details on the construction and properties of $W^r_d$, we refer the reader to \cite{ACGH85,Fla11}. 

Consider now the locally closed subschemes $\omega^r := W^r_d \setminus W^{r+1}_d$. By \cite[Lemma IV.3.5]{ACGH85}, we have that $\omega^{r+1}$ is contained in the closure of $\omega^r$.  
Let $\omega^r_P:= \widetilde{\alpha}_{P,d}(\omega^r)$. It follows from our observations that the disjoint union $\coprod_r \omega^r_P$ defines a finite stratification of $W_d$.

\begin{lemma}\label{ansymwd} The map $\alpha^{(d),an}_P:\Sym^{d,an}(X) \rightarrow W^{an}_d$ is a homotopy equivalence.
\end{lemma}
\begin{proof} Let $\Sigma^r \subset \Sym^d(X)$ denote the preimage of $\omega^r_P$. Given $x \in \omega^r_P$, let $[D]$ denote the corresponding divisor class in $\omega^r$, and let $D\in[D]$. Then the preimage of $x$ is equal to the complete linear series $|D| \subset \Sym^d(X)$. In particular $\Sigma^r \rightarrow \omega^r_P$ is a projective bundle of rank $r$. Therefore, $\Sym^d(X) \rightarrow W_d$ satisfies property $(\dagger)$ from \S \ref{subsectfiber}. By Theorem \ref{projective fibers}, the map $\Sym^{d,an}(X) \rightarrow W^{an}_d$ is a homotopy equivalence.
\end{proof}

As a consequence of Proposition \ref{symabeljac}, the map $r^J:J^{an} \rightarrow J(\Gamma)$ restricts itself to a map from $W^{an}_d$ to the image $W_d(\Gamma)$ of $\Sym^d(\Gamma)$. We now establish the following theorem.
{\renewcommand{\thetheorem}{\ref{tropicalize wd}} 
\begin{theorem} The map $W^{an}_d \rightarrow W_d(\Gamma)$ is a homotopy equivalence.
\end{theorem}
\addtocounter{theorem}{-1}}
\begin{proof} From Proposition \ref{symabeljac}, we have the following commutative diagram.
	\begin{figure}[H]
		\centering
		\includegraphics[scale=1]{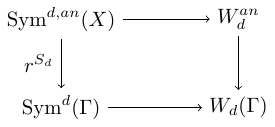}
	\end{figure}
\noindent{}By Lemma \ref{ansymwd}, the top arrow is a homotopy equivalence. By Theorem \ref{symskel}, the left arrow is a homotopy equivalence. Finally, by Proposition \ref{trophomotopy}, the bottom arrow is a homotopy equivalence, and thus the right arrow is a homotopy equivalence.
\end{proof}


\subsection{Proof of non-Archimedean Lefschetz} We now establish our Lefschetz hyperplane theorem for non-Archimedean Jacobians.

{\renewcommand{\thetheorem}{\ref{nonarch wd}} 
\begin{theorem} For $1 \leq d\leq g-1$, the pair $(J^{an},W^{an}_d)$ is $d$-connected.
\end{theorem} 
\addtocounter{theorem}{-1}}

\begin{proof} Consider the following diagram.
	\begin{figure}[H]
		\centering
		\includegraphics[scale=1]{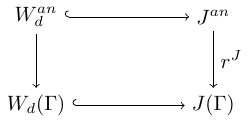}
		\qquad
	\end{figure}
\noindent{}By Theorem \ref{tropicalize wd}, both vertical arrows are homotopy equivalences. By Theorem \ref{tropical wd}, the pair $(J(\Gamma),W_d(\Gamma))$ is $d$-connected, and the theorem follows.
\end{proof}


\subsection*{Acknowledgement} I would like to thank Sam Payne for suggesting the problem and for his guidance and support. I am also grateful to Tyler Foster, Dhruv Ranganathan, Arseniy Sheydvasser and the referee for helpful comments and corrections. Finally, I benefited from many fruitful conversations with Melody Chan, David Jensen and Johannes Nicaise.

\bibliographystyle{abbrv}
\bibliography{lefschetz}

\end{document}